\documentclass[reqno]{amsart}
\usepackage{amsmath}
\usepackage{amssymb}
\usepackage{amsthm}
\usepackage{hyperref}
\hypersetup{backref=true}
\usepackage{cases}

\title{A generalization of partition identities of G\"ollnitz-Gordon, Rogers-Ramanujan and Nandi}

\author{Motoki Takigiku}
\address{Graduate School of Natural Science and Technology, Okayama University, Okayama 700-8530, Japan (current affiliation : Dharmacapital Pte. Ltd.)}

\author{Shunsuke Tsuchioka}
\address{Department of Mathematical and Computing Sciences, Institute of Science Tokyo, Tokyo 152-8550, Japan}
\email{tshun@kurims.kyoto-u.ac.jp}

\thanks{M.T.\,was supported by Start-up research support from Okayama University.
S.T.\,was supported by the Research Institute for Mathematical
Sciences, an International Joint Usage/Research Center located in Kyoto
University, the TSUBAME3.0 supercomputer at Tokyo Institute of Technology,
the TSUBAME4.0 supercomputer at Institute of Science Tokyo,
JSPS Kakenhi Grant 20K03506, 23K03051, 26K06722, JST CREST Grant Number JPMJCR2113, Japan, 
Leading Initiative for Excellent Young Researchers, MEXT, Japan
and JST FOREST Program (Grant Number JPMJFR246C, Japan).}

\date{Jul 2, 2026}

\usepackage{tikz}

\def\node#1#2{\overset{#1}{\underset{#2}{\circ}}}

\def\ver#1#2{\overset{{\llap{$\scriptstyle#1$}\displaystyle\circ{\rlap{$\scriptstyle#2$}}}}{\scriptstyle\vert}}

\newtheorem{Thm}{Theorem}[section]

\newtheorem{Conj}[Thm]{Conjecture}

\newtheorem{Prop}[Thm]{Proposition}

\newtheorem{Rem}[Thm]{Remark}

\newtheorem{Ex}[Thm]{Example}

\newcommand{\BETAUNIT}{\beta^{\mathsf{U}}}
\newcommand{\ALPHAUNIT}{\alpha^{\mathsf{U}}}
\newcommand{\KYU}[2]{Q_{#1}^{(#2)}}
\newcommand{\ERU}[3]{R_{#1,#2}^{(#3)}}
\newcommand{\KYUU}[1]{Q_{#1}}
\newcommand{\ERUU}[2]{R_{#1,#2}}
\newcommand{\LAMBDA}[1]{\Lambda'_{#1}}
\newcommand{\qbinom}{\genfrac{[}{]}{0pt}{}}
\newcommand{\GG}{\mathsf{GG}}
\newcommand{\HH}{H}
\newcommand{\NAN}{\mathcal{N}}
\newcommand{\PC}{\mathcal{C}}
\newcommand{\PD}{\mathcal{D}}
\newcommand{\PT}{\stackrel{\mathsf{PT}}{\sim}}
\DeclareMathOperator{\PAR}{\mathsf{Par}}
\DeclareMathOperator{\ZZ}{\mathsf{Z}}
\newcommand{\RR}{\mathsf{RR}}

\begin{document}
\maketitle

\begin{abstract}  
We propose Andrews-Gordon type series for certain level 2 standard modules
of type $A^{(2)}_{\textrm{odd}}$, and
prove
the corresponding sum-product identities 
except for $A^{(2)}_{6n+3}$.
These identities generalize the identities of G\"ollnitz-Gordon (mod 8),
Rogers-Ramanujan (mod 5) and (partially) Nandi (mod 14).
\end{abstract}

\section{Introduction}
In this paper, we adopt the usual convention on the $q$-Pochhammer symbol
$(a_1,\dots,a_m;q)_{n}=\prod_{k=1}^{m}\prod_{c=1}^{n}(1-a_kq^{c-1})$
for $n\in\mathbb{Z}_{\geq 0}\sqcup\{\infty\}$ and $m\geq 1$.
As usual, the symbol $\delta_{m,n}$ stands for the Kronecker delta.

\subsection{The Lepowsky-Wilson program}
Since the vertex operator proof of the Rogers-Ramanujan identities
by Lepowsky and Wilson~\cite[\S10]{LW3}, 
it has been expected that, for an integrable highest weight module
of highest weight $\lambda$ of the affine Lie algebra of type $X^{(r)}_N$, 
the principal character $\chi_{X^{(r)}_N}(\lambda)$ (see ~\cite[\S7]{LW3}) should admit
Rogers-Ramanujan type infinite sum expansions.
In the companion paper~\cite[Theorem 14.4]{LW4}, Lepowsky and Wilson gave
a vertex operator interpretation of the
Andrews-Gordon-Bressoud identities, recalled as Theorem \ref{AGTHM} below. 
The case $\ell=\varepsilon=1$ recovers the Rogers-Ramanujan identities.
\begin{Thm}[{\cite[\S3.2]{Sil}}]\label{AGTHM}
Let $\varepsilon=0,1$ and $\ell\geq 1$. For $1\leq b\leq \ell+1$, we have
\begin{align*}
\sum_{j_1,\dots,j_{\ell}\geq 0}
\frac{q^{J_1^2+\dots+J_{\ell}^2+J_{b}+\dots+J_{\ell}}}{(q;q)_{j_1}\cdots(q;q)_{j_{\ell-1}}(q^{2-\varepsilon};q^{2-\varepsilon})_{j_{\ell}}}
=\frac{(q^b,q^{2\ell+\varepsilon+2-b},q^{2\ell+\varepsilon+2};q^{2\ell+\varepsilon+2})_{\infty}}{(q;q)_{\infty}}.
\end{align*}
Here, we put $J_k=j_k+\dots+j_{\ell}$ for $1\leq k\leq \ell$. When $b=\ell+1$, the sum  $J_b+\dots+J_{\ell}$ is understood to be empty (and is defined to be 0).
\end{Thm}

Note that the right hand side is equal to
$\chi_{A^{(1)}_{1}}((2\ell+\varepsilon+1-b)\Lambda_0+(b-1)\Lambda_1)$ (see Figure \ref{dynk}).
For the current status of the Lepowsky-Wilson program,
see a brief survey in ~\cite[\S1.4]{Tsu2}.

\subsection{The main result}
In this paper, we focus on level 2 of type $A^{(2)}_{2\ell-1}$ (see Figure \ref{dynk}) for $\ell\geq 3$.
Any level 2 dominant integral weight of $A^{(2)}_{2\ell-1}$ is given by
$\LAMBDA{b}=(\delta_{b,0}+\delta_{b,1})\Lambda_0+\Lambda_{b}$ for some $0\leq b\leq \ell$ modulo the Dynkin diagram automorphism, and 
the principal character of the corresponding
standard module is given by
\begin{align*}
  \chi_{A^{(2)}_{2\ell-1}}(\LAMBDA{b})
  =
  \frac{(q^{2\ell+2};q^{2\ell+2})_{\infty}}{(q^2;q^2)_{\infty}}
  \frac{(q^{2b},q^{2\ell+2-2b};q^{2\ell+2})_{\infty}}{(q^{b},q^{2\ell+2-b};q^{2\ell+2})_{\infty}}
  = \frac{Q(q^{2\ell+2},q^b)}{(q^2;q^2)_{\infty}}
\end{align*}
for $1\leq b\leq \ell$. Here, $Q$ is the quintuple product~\cite[(1.74)]{Sil} 
\begin{align}
  Q(w,x)= (-w/x,-x,w;w)_{\infty}(w/x^2,wx^2;w^2)_{\infty}.
\label{WQ}
\end{align}

Our main result gives Rogers-Ramanujan type identities
whose infinite products are the principal characters
of certain level 2 standard modules of type $A^{(2)}_{6n\pm 1}$.

\begin{Thm}\label{MAINA2oddN3}
  Let $\varepsilon=0,1$ and $n\geq 1$. For $0\leq b\leq 2n+1-\varepsilon$, we have
  \begin{align*}
    \sum_{i,j_1,\dots,j_{2n-\varepsilon}\geq 0}
    \frac{q^{\KYU{n}{\varepsilon}(i,j_1,\dots,j_{2n-\varepsilon})+\ERU{n}{b}{\varepsilon}(i,j_1,\dots,j_{2n-\varepsilon})}}{(q^2;q^2)_i(q^2;q^2)_{j_1}\cdots(q^2;q^2)_{j_{2n-\varepsilon}}}
      =\chi_{A^{(2)}_{6n+(-1)^{\varepsilon}}}(\LAMBDA{n+b}).
  \end{align*}
  Here, we put $J_k=j_k+\dots+j_{2n-\varepsilon}$ for $1\leq k\leq 2n-\varepsilon$ and
\begin{align*}
    \KYU{n}{\varepsilon}(i,j_1,\dots,j_{2n-\varepsilon})
    &=
    (n+1-\varepsilon)i^2+2(J_1^2+\dots+J_{2n-\varepsilon}^2)+
    2i(J_1+\dots+J_{2n-\varepsilon}),\\
    \ERU{n}{b}{\varepsilon}(i,j_1,\dots,j_{2n-\varepsilon})
    &=
    (b-1+\varepsilon)i+2(J_{2n-\varepsilon+2-b}+J_{2n-\varepsilon+3-b}+\dots+J_{2n-\varepsilon}).
  \end{align*}
\end{Thm}
When $b=0,1$, the sum $J_{2n-\varepsilon+2-b}+J_{2n-\varepsilon+3-b}+\dots+J_{2n-\varepsilon}$ in $\ERU{n}{b}{\varepsilon}$ is empty. Note that $6n+(-1)^{\varepsilon}=2\ell-1$ for $\ell=3n+1-\varepsilon$.
As we will see in \S\ref{classi}, our main result above
recovers the partition identities of
G\"ollnitz-Gordon, Rogers-Ramanujan and (partially) Nandi
when $(n,\varepsilon)=(1,1),(1,0),(2,1)$ respectively (see Remark \ref{recov}).
\begin{figure}\label{dynk}
\begin{align*}
\begin{array}{ll@{\qquad\qquad}ll}
A_1^{(1)} & \node{1}{\alpha_0} \Leftrightarrow \node{1}{\alpha_1} &
  A_{2\ell-1}^{(2)}  &\node{}{\alpha_1}-\node{\ver{}{\alpha_0}}{\alpha_2}\!\!{}-\node{}{\alpha_3}-\cdots-\node{}{\alpha_{\ell-1}}\!\!\!\Leftarrow\!\node{}{\alpha_\ell}
\end{array}
\end{align*}
\caption{The affine Dynkin diagrams $A^{(1)}_1$ and $A^{(2)}_{2\ell-1}$}
\end{figure}

\subsection{Future directions}
The fact that the left hand side in
Theorem \ref{MAINA2oddN3} is given as
a coherent family with respect to $b$ plays a key role in the proof (see \S\ref{mainproof}). 
This suggests a method for proving Rogers-Ramanujan type identities
arising from a family of affine Lie algebras. Namely, to prove a
Rogers-Ramanujan type identity,
it is sometimes useful to embed it into a family of identities.
\begin{Conj}\label{MAINA2oddConj}
  Let $n\geq 1$. For $0\leq b\leq 2n$, we have
  \begin{align*}
    \sum_{i,j,k,\ell_1,\dots,\ell_{2n}\geq 0}
    \frac{q^{\KYUU{n}(i,j,k,\ell_1,\dots,\ell_{2n})+\ERUU{n}{b}(i,j,k,\ell_1,\dots,\ell_{2n})}}{(q^2;q^2)_i(q^6;q^6)_j(q^6;q^6)_k(q^2;q^2)_{\ell_1}\cdots(q^2;q^2)_{\ell_{2n}}}
      =\chi_{A^{(2)}_{6n+3}}(\LAMBDA{n+2+b}).
  \end{align*}
  Here, we put $L_k=\ell_k+\dots+\ell_{2n}$ for $1\leq k\leq 2n$ and
\begin{align*}
    \KYUU{n}(i,j,k,\ell_1,\dots,\ell_{2n})
    &=
    n(i+3j+3k)^2+2(L_1^2+\dots+L_{2n}^2)+
    2(i+3j+3k)(L_1+\dots+L_{2n})+2iL_{2n},\\
    \ERUU{n}{b}(i,j,k,\ell_1,\dots,\ell_{2n})
    &=
    (b+2)i+(3b+2)j+(3b+4)k+2(L_{2n+1-b}+L_{2n+2-b}+\dots +L_{2n}).
  \end{align*}
\end{Conj}
When $b=0$, the sum $L_{2n+1-b}+L_{2n+2-b}+\dots +L_{2n}$ in $\ERUU{n}{b}$ is empty.
For a possible proof strategy, see Remark \ref{obs}.
As we will see in \S\ref{classi}, our conjecture above
recovers the partition identities of
Kanade-Russell-Bringmann-Jennings-Shaffer-Mahlburg-Rosengren 
when $n=1$ (see Remark \ref{KRanotherproof}).
For relationships between
the Lepowsky-Wilson program for $A^{(2)}_{\textrm{odd}}$ level 2
and the aforementioned partition identities, see Remark \ref{interp}.
It would also be interesting to compare
Theorem \ref{MAINA2oddN3} with the conjectural Rogers-Ramanujan type identities in ~\cite[\S6]{TT2}
whose infinite products are given by the principal characters of level 2 modules of type $A^{(2)}_{13}$.

\hspace{0mm}

\noindent{\bf Organization of the paper.} 
In \S\ref{pre}, we review standard tools in $q$-series, and give a proof of
Theorem \ref{MAINA2oddN3} in \S\ref{mainproof}.
In \S\ref{classi}, we briefly 
describe automatic calculations that connect the relevant partition conditions and Andrews-Gordon type series.

\section{Preparations}\label{pre}
For $m,n\in\mathbb{Z}$, we put $\qbinom{n}{m}_{q} = \frac{(q;q)_n}{(q;q)_m(q;q)_{n-m}}$ if $0\leq m\leq n$, and $\qbinom{n}{m}_q=0$ otherwise.

\subsection{The Bailey lattice}
In the following, Theorem \ref{BLEMMA} and Theorem \ref{BCHANGE} are referred to only in Remark \ref{obs}, and
are not needed in the proof of Theorem \ref{MAINA2oddN3}.

Recall that a Bailey pair relative to $a$ is a pair of
sequences $((\alpha_n)_{n\geq 0},(\beta_n)_{n\geq 0})$ that satisfies
\begin{align*}
\beta_n=\sum_{r=0}^{n}\frac{\alpha_r}{(q;q)_{n-r}(aq;q)_{n+r}}.
\end{align*}
for $n\geq 0$.
In this paper, we always assume $\alpha_0=\beta_0=1$ (see ~\cite[\S10.1]{ML}). 

\begin{Ex}[{\cite[\S2.3.6]{Sil}, \cite[(11.11)]{ML}}]\label{UBP}
  The unit Bailey pair 
  \begin{align*}
    \BETAUNIT_n=\delta_{n,0},\quad
    \ALPHAUNIT_n=(-1)^nq^{n\choose 2}\frac{1-aq^{2n}}{1-a}\frac{(a;q)_n}{(q;q)_n}
  \end{align*}
  is a Bailey pair relative to $a$.
\end{Ex}


\begin{Thm}[{\cite[Theorem 9.1]{Cha}, \cite[Corollary 11.1]{ML}}]\label{BLEMMA}
If $((\alpha_n)_{n\geq 0},(\beta_n)_{n\geq 0})$ is a Bailey pair relative to $a$,
then so is $((\alpha'_n)_{n\geq 0},(\beta'_n)_{n\geq 0})$, where
\begin{align*}
  \alpha'_n=a^nq^{n^2}\alpha_n,\quad
  \beta'_n=\sum_{r=0}^{n}\frac{a^rq^{r^2}}{(q;q)_{n-r}}\beta_r.
\end{align*}
\end{Thm}



\begin{Thm}[{\cite[Theorem 3.2]{War}, \cite[Theorem 13.1]{ML}}]\label{BCHANGE}
If $((\alpha_n)_{n\geq 0},(\beta_n)_{n\geq 0})$ is a Bailey pair relative to $aq^N$ for $N\geq 0$,
then $((\alpha^{\ast}_n)_{n\geq 0},(\beta_n)_{n\geq 0})$
is a Bailey pair relative to $a$, where
\begin{align*}
  \alpha^{\ast}_n=(1-aq^{2n})(aq;q)_N\sum_{j=0}^{n}(-a)^jq^{2nj-{j+1\choose 2}}
  \qbinom{N}{j}_q\frac{1}{(aq^{2n-j};q)_{N+1}}\alpha_{n-j}.
\end{align*}
\end{Thm}

\begin{Thm}[{\cite[Corollary 4.2]{AAB}, \cite[Theorem 11.4]{ML}}]\label{BLATTICE}
  Let $((\alpha_n)_{n\geq 0},(\beta_n)_{n\geq 0})$ be a Bailey pair relative to $a$.
  For $k\geq 1$ and $0\leq j\leq k$, we have
  \begin{align*}
    {} &{}
    \sum_{J_1\geq \dots\geq J_k\geq 0}
    \frac{q^{J_1^2+\dots+J_k^2-(J_1+\dots+J_j)}a^{J_1+\dots+J_k}}{(q;q)_{J_1-J_2}\dots(q;q)_{J_{k-1}-J_k}}
      \beta_{J_k}\\
      &=
      \frac{1}{(aq;q)_{\infty}}\sum_{t\geq 0}\frac{a^{kt}q^{kt^2-jt}(1-a^{j+1}q^{2t(j+1)})}{(1-aq^{2t})}
      \alpha_{t}.
  \end{align*}
\end{Thm}


\subsection{The quintuple product}
Recall \eqref{WQ} and Watson's quintuple product identity (see ~\cite[\S1.7.3]{Sil})
\begin{align*}
  Q(w,x)
  &= f(-wx^3,-w^2x^{-3})+xf(-wx^{-3},-w^2x^3)\\
  &= \sum_{\ell\in\mathbb{Z}}(-1)^{\ell}w^{\ell(3\ell-1)/2}(x^{3\ell}+x^{-3\ell+1})
\end{align*}
in terms of Jacobi's triple product identity and
Ramanujan's theta (see ~\cite[\S1.7.2]{Sil})
\begin{align*}
  f(c,d)=(-c,-d,cd;cd)_{\infty}=\sum_{n\in\mathbb{Z}}c^{n(n+1)/2}d^{n(n-1)/2}.
\end{align*}

\begin{Prop}\label{QUIN}
  For $\delta=\pm1$, the quintuple product
   $Q(V^6q^{2\delta},UVq^{\delta})$ is equal to
  \begin{align*}
    \sum_{i,t\geq 0}\qbinom{t+i}{t}_{q^2}
    (-1)^t q^{\delta(i+t^2+t+(1-\delta)ti)}
    V^{(i+2t)^2}U^{(i+2t)}(1-(V^2/U)^{2(2t+i+1)}).
  \end{align*}
\end{Prop}

\begin{proof}
Note $V^{(i+2t)^2}U^{(i+2t)}(V^2/U)^{2(2t+i+1)}=V^{(i+2t+2)^2}U^{-(i+2t+2)}$ and
\begin{align*}
  Q(V^6q^{2\delta},UVq^{\delta})
  =
  \sum_{\ell\in\mathbb{Z}}(-1)^{\ell}(
  V^{9\ell^2}U^{3\ell}q^{\delta\ell(3\ell+2)}
  +V^{(1-3\ell)^2}U^{1-3\ell}q^{\delta(1-3\ell)(1-\ell)}).
\end{align*}

Thus, for $m\geq 0$ and $m'<0$, it is enough to show 
\begin{align*}
  \sum_{\substack{i,t\geq 0 \\ i+2t=m}}\qbinom{t+i}{t}_{q^2}
  (-1)^t q^{\delta(i+t^2+t+(1-\delta)ti)}
  &=
  \begin{cases}
    (-1)^M q^{\delta M(3M+2)} & \textrm{if $m=3M$}, \\
    (-1)^M q^{\delta(1+3M)(1+M)} & \textrm{if $m=3M+1$}, \\
    0 & \textrm{if $m=3M+2$},
  \end{cases}\\  
  \sum_{\substack{i,t\geq 0 \\ i+2t+2=-m'}}\qbinom{t+i}{t}_{q^2}
  (-1)^{t+1} q^{\delta(i+t^2+t+(1-\delta)ti)}
  &=
  \begin{cases}
    (-1)^{M'} q^{\delta M'(3M'-2)} & \textrm{if $m'=-3M'$}, \\
    (-1)^{M'} q^{\delta (1-3M')(1-M')} & \textrm{if $m'=-3M'+1$}, \\
    0 & \textrm{if $m'=-3M'+2$}.
  \end{cases}  
\end{align*}
Because the latter is easily obtained by the former, we show the former.
The case $\delta=1$ follows easily from ~\cite[(16)]{War2} (see also Remark \ref{autopr} for an automatic proof).
The case $\delta=-1$ follows from the case $\delta=1$ by virtue of $\qbinom{t+i}{t}_{q^{-2}}=q^{-2ti}\qbinom{t+i}{t}_{q^{2}}$.
\end{proof}


\section{Proof of Theorem \ref{MAINA2oddN3}}\label{mainproof}
Apply Theorem \ref{BLATTICE} to the unit Bailey pair (Example \ref{UBP}), and set
\begin{align*}
A=(\textrm{the left hand side of Theorem \ref{BLATTICE}})|_{a=q^{i+1},k=2n+1-\varepsilon,j=2n+1-\varepsilon-b,(\beta_n)_{n\geq 0}=(\BETAUNIT_n)_{n\geq 0}}
\end{align*}
so that
the infinite sum in Theorem \ref{MAINA2oddN3} is equal to
\begin{align*}
  \sum_{i\geq 0}(A|_{q=q^2})\frac{q^{(n+1-\varepsilon)i^2+(b-1+\varepsilon)i}}{(q^2;q^2)_i}.
\end{align*}


Let $a=q^{i+1},k=2n+1-\varepsilon,j=2n+1-\varepsilon-b, (\beta_n)_{n\geq 0}=(\BETAUNIT_n)_{n\geq 0}$ (and $(\alpha_n)_{n\geq 0}=(\ALPHAUNIT_n)_{n\geq 0}$) as above.
By Theorem \ref{BLATTICE}, we have
\begin{align*}
  A=\frac{1}{(q^{i+2};q)_{\infty}}\sum_{t\geq 0}
  q^{t(tk+(i+1)k-j)}\frac{1-q^{(j+1)(2t+i+1)}}{1-q^{2t+i+1}}
  (-1)^tq^{{t\choose 2}}\frac{(1-q^{2t+i+1})}{1-q^{i+1}}\frac{(q^{i+1};q)_t}{(q;q)_t}.
\end{align*}

Thus, the infinite sum in Theorem \ref{MAINA2oddN3} is equal to
$B/(q^2;q^2)_{\infty}$, where
\begin{align*}
  B=
  \sum_{i,t\geq 0}\qbinom{t+i}{t}_{q^2}
  (-1)^t q^{(n+1-\varepsilon)i^2+(b-1+\varepsilon)i+t(2(tk+(i+1)k-j)+t-1)}
  (1-q^{2(j+1)(2t+i+1)}).
\end{align*}
Let $T=q^{n+1-\varepsilon}$ and $S=q^{b-\varepsilon}$ so that $q^{j+1}=T^2/S$.
It is easy to see
\begin{align*}
B=\sum_{i,t\geq 0}\qbinom{t+i}{t}_{q^2}
(-1)^t q^{(2\varepsilon-1)i+t(t-1)+2(\varepsilon-1)t(t+i+1)+2t}
T^{(i+2t)^2}S^{(i+2t)}(1-(T^2/S)^{2(2t+i+1)}).
\end{align*}
Let $\delta=2\varepsilon-1$ for $\varepsilon=0,1$. By Proposition \ref{QUIN}, we have
$B = Q(T^6q^{4\varepsilon-2},STq^{2\varepsilon-1}) = Q(q^{6n+(-1)^{\varepsilon}+3},q^{n+b})$, 
which proves Theorem \ref{MAINA2oddN3}.

\begin{Rem}\label{obs}
A similar proof strategy may be applicable to Conjecture \ref{MAINA2oddConj}. 
Here is a sketch. Let $N\geq 0$ and take
the unit Bailey pair (Example \ref{UBP})
$((\ALPHAUNIT_n)_{n\geq 0},(\BETAUNIT_n)_{n\geq 0})$
relative to $aq^{N}$.
Applying the weak Bailey Lemma (Theorem \ref{BLEMMA}) twice, we get a Bailey pair
$((\alpha_n)_{n\geq 0},(\beta_n)_{n\geq 0})$
relative to $aq^N$, where
\begin{align*}
\alpha_n=a^{2n}q^{2n^2+2Nn}\ALPHAUNIT_n(aq^N,q),\quad
\beta_n=\sum_{r=0}^{n}\frac{a^rq^{r^2+Nr}}{(q;q)_{n-r}(q;q)_r}.
\end{align*}

By a base change (Theorem \ref{BCHANGE}), we get a Bailey pair
$((\alpha^{\ast}_n)_{n\geq 0},(\beta_n)_{n\geq 0})$ relative to $a$.
Then, the left hand side of Theorem \ref{BLATTICE} is nothing but 
\begin{align*}
\sum_{J_1\geq \dots\geq J_k\geq J_{k+1}\geq 0}
  \frac{q^{J_1^2+\dots+J_k^2+J_{k+1}^2+NJ_{k+1}-(J_1+\dots+J_j)}a^{J_1+\dots+J_k+J_{k+1}}}{(q;q)_{J_1-J_2}\dots(q;q)_{J_{k-1}-J_k}(q;q)_{J_{k}-J_{k+1}}(q;q)_{J_{k+1}}}.
\end{align*}

Based on this observation, 
by an argument similar to that in the proof of Theorem \ref{MAINA2oddN3} (omitting the tedious calculation),
Conjecture \ref{MAINA2oddConj} for $1\leq b\leq 2n$ is reduced to showing
\begin{align*}
  {} &\sum_{\substack{v,t,u,p,r\geq 0 \\ v+3t+3u+3p+2r=m}}
  (-1)^{p+r}
  q^{v+t+2u+r(3v+3p+3t+3u+1)+p(p+1)/2+r(3r-1)/2}\\
  &\quad\quad (1-q^{2r+2v+2p+3t+3u+1})\frac{(q;q)_{2p+2r+v+3t+3u}}
       {(q;q)_{3p+2r+2v+3t+3u+1}}\frac{(q;q)_{r+2v+2p+3t+3u}}{(q^3;q^3)_t(q^3;q^3)_u(q;q)_{r}(q;q)_p(q;q)_{v}}\\
       &=
       \begin{cases}
         (-1)^Mq^{3M(3M+1)/2} & \textrm{if $m=3M$},\\
         (-1)^Mq^{(3M+2)(3M+1)/2} & \textrm{if $m=3M+1$},\\
         0 & \textrm{if $m=3M+2$}.
       \end{cases}
\end{align*}
for $m\geq 0$. Because the summand is $q$-proper hypergeometric,
$q$-Sister Celine technique ~\cite[\S5]{WZ} should work (i.e., provide the desired recurrence relation) in principle (of course, we need to be rigorous on
the subtle details, such as confirmation of the admissibility~\cite[\S5.2]{WZ}).
We hope this observation will be useful to future works.
\end{Rem}

\begin{Rem}\label{autopr}
As a proof of concept of the reduction in Remark \ref{obs}, we present an automatic proof of
the identity in the proof of Proposition \ref{QUIN}. Let
\begin{align*}
  F_k(M,t) =
  \qbinom{3M+k-t}{t}_{q^2}(-1)^tq^{(3M+k-2t)+t^2+t})/((-1)^M q^{(3M+2-k)(M+k)})
\end{align*}
for $M\geq 0$ and $k=0,1,2$ in virtue of the Wilf trick.
Our aim is to show
\begin{align}
\sum_{t\in\mathbb{Z}}F_{0}(M,t)=1=
\sum_{t\in\mathbb{Z}}F_{1}(M,t),\quad
\sum_{t\in\mathbb{Z}}F_{2}(M,t)=0.
\label{saigoQUIN}
\end{align}

A straightforward calculation shows that
\begin{align*}
R_{k}(M,t)=\frac{q^{4t-(6M+2+2k)}(1-q^{2t})(-q^{6t}+q^{12M+10+4k})(q^{6M+2+2k}-q^{2t})}{(-q^{6M+6+2k}+q^{4t})(-q^{6M+4+2k}+q^{4t})(-q^{6M+2+2k}+q^{4t})}
\end{align*}
is a Wilf-Zeilberger certificate~\cite{WZ2} (see also ~\cite[\S1.4]{Sil} and ~\cite[\S9.2]{Cha}), i.e., 
we have
\begin{align*}
F_{k}(M+1,t)-F_{k}(M,t)=G_{k}(M,t+1)-G_{k}(M,t),
\end{align*}
for $G_{k}(M,t)=F_{k}(M,t)R_{k}(M,t)$.
As usual, it implies that $\sum_{t\in\mathbb{Z}}F_{k}(M,t)$ depends only on $M$.
Thus, we have \eqref{saigoQUIN} for $M\geq 0$
by verifying \eqref{saigoQUIN} for $M=0$.
\end{Rem}

\section{Relations to partition identities}\label{classi}

Let $\PAR(n)$ (resp. $\PAR$) denote the set of partitions of $n$ (resp. partitions).
Recall that two subsets $\PC$ and $\PD$ of $\PAR$ are
partition theoretically equivalent~\cite[Definition 3]{An0} (abbreviated to $\PC\PT\PD$)
if we have $|\PC\cap\PAR(n)|=|\PD\cap\PAR(n)|$ for $n\geq 0$.

We briefly review the relevant partition theorems.
As usual, for a partition $\lambda$ and positive integer $j$, we denote by
$m_j(\lambda)$ the number of parts of $\lambda$ that are equal to $j$.
For $i=1,2$, we denote by $\GG_i$ (resp. $\RR_i$) the set of
G\"ollnitz-Gordon~\cite[\S2.4.4]{Sil} (resp. Rogers-Ramanujan) partitions, namely,
\begin{align*}
  \RR_i &= \{\lambda\in\PAR\mid\textrm{$\lambda_k-\lambda_{k+1}\geq 2$ for $1\leq k<\ell(\lambda)$, and $m_1(\lambda)\leq 2-i$}\},\\
  \GG_1 &= \{\lambda\in\RR_1\mid\textrm{$\lambda_k-\lambda_{k+1}=2$ for $1\leq k<\ell(\lambda)$ implies $\lambda_k$ is odd}\},\\
  \GG_2 &= \{\lambda\in\GG_1\mid m_1(\lambda)=m_2(\lambda)=0\}.
\end{align*}

The 
G\"ollnitz-Gordon and Rogers-Ramanujan
partition theorems are stated as
\begin{align*}
  \GG_i\PT T^{(8)}_{2i-1,4,9-2i},\quad
  \RR_i\PT T^{(5)}_{i,5-i},
\end{align*}
where $T^{(N)}_{a,\dots,b}=\{\lambda\in\PAR\mid\textrm{$\lambda_k\equiv a,\dots,b\pmod{N}$ for $1\leq k\leq\ell(\lambda)$}\}$.

Moreover, for $j=1,2,3$, let $\HH_j$ (resp. $\NAN_j$)
be the set of
Kanade-Russell~\cite[3.1.$j$]{KR2} (resp. Nandi~\cite[Conjecture 5.$(4+j)$]{Sil}) partitions (see also ~\cite{Nan,TT}).
We do not duplicate the definitions because, unlike $\GG_i$ and $\RR_i$, their definitions are not stated in a single line.
We refer to the following partition theorem
\begin{align}
  \HH_1\PT T^{(12)}_{1,4,6,8,11},\quad
  \sum_{\lambda\in\HH_2}q^{|\lambda|}
  =\frac{(q^6;q^{12})_{\infty}}{(q^2,q^3,q^4,q^8,q^9,q^{10};q^{12})_{\infty}},\quad
  \HH_3\PT T^{(12)}_{4,5,6,7,8}.
\label{KR12}
\end{align}
as ``Kanade-Russell conjecture modulo 12''
although it is a theorem due to Bringmann et al. and Rosengren~\cite{BJM,Ros}.
By ~\cite{TT}, we also have
\begin{align*}
  \NAN_1\PT T^{(14)}_{2,3,4,10,11,12},\quad
  \NAN_2\PT T^{(14)}_{1,4,6,8,10,13},\quad
  \NAN_3\PT T^{(14)}_{2,5,6,8,9,12}.
\end{align*}

\begin{Rem}\label{interp}
For an affine Dynkin diagram $A^{(2)}_{2\ell-1}$, where $\ell\geq 3$,
 and a level 2 dominant integral weight $\LAMBDA{b}=(\delta_{b,0}+\delta_{b,1})\Lambda_0+\Lambda_{b}$, where $0\leq b\leq \ell$, we consider
 the standard module $V(\LAMBDA{b})$ with a highest weight vector
 $w_{\LAMBDA{b}}$. 
 Let $\ZZ_{k}(\beta)$, where $k\in\mathbb{Z}$, be a Lepowsky-Wilson $Z$-operator~\cite{LW3} (see a review in ~\cite[\S2]{Tsu2})  associated with a root $\beta$ of $A_{2\ell-1}$, whose simple roots $\{\alpha_1,\dots,\alpha_{2\ell-1}\}$ are taken as usual (see ~\cite{Ito}).

 It is known by
  ~\cite[Theorem 1.2, Theorem 1.3, Theorem 1.4]{Ito} that 
  \begin{enumerate}
  \item $\{\ZZ_{-\lambda_1}\cdots\ZZ_{-\lambda_{\ell}}w_{\Lambda'_{2i-1}}\mid \lambda\in\GG_i\}$ spans $V(\Lambda'_{2i-1})$ for $\ell=3$ and $i=1,2$,
  \item $\{\ZZ_{-\lambda_1}\cdots\ZZ_{-\lambda_{\ell}}w_{\Lambda'_{2i-1}}\mid \lambda\in\RR_i\}$ spans $V(\Lambda'_{2i-1})$ for $\ell=4$ and $i=1,2$,
  \item $\{\ZZ_{-\lambda_1}\cdots\ZZ_{-\lambda_{\ell}}w_{\Lambda'_{2j-1}}\mid \lambda\in\HH_j\}$ spans $V(\Lambda'_{2j-1})$ for $\ell=5$ and $j=1,2,3$,
  \end{enumerate}
  where $\ZZ_k=\ZZ_k(\alpha_1)$ (resp. $\ZZ_k=\ZZ_k(\alpha_{\ell})$) if $k$ is even (resp. odd). Note that, by the partition theorems, we may replace the word ``spans'' with ``is a basis of''.

  It is reasonable to expect the same form of vertex operator interpretation  of Nandi's partitions $\NAN_j$, whose definition originates from $A^{(2)}_2$ level 4~\cite{Nan},  via $A^{(2)}_{11}$ level 2 (see also ~\cite[Theorem 1.1]{Ito} and ~\cite[\S1.1]{KR2}). 
\end{Rem}


\begin{Prop}\label{computerproof}
  For $a=1,2$, we have
\begin{align*}
\sum_{\lambda\in\GG_a}x^{\ell(\lambda)}q^{|\lambda|}
&=
\sum_{s,t\geq 0}
\frac{q^{Q^{(1)}_1(s,t)+\delta_{a,2}2(s+t)}}{(q^2;q^2)_s(q^2;q^2)_t}x^{s+t},\\
\sum_{\lambda\in\RR_a}x^{\ell(\lambda)}q^{|\lambda|}
&=
\sum_{s,t,u\geq 0}\frac{q^{Q^{(0)}_{1}(s,t,u)+(2a-3)s+(2a-2)u}}{(q^2;q^2)_s(q^2;q^2)_t(q^2;q^2)_u}x^{s+t+2u},\\
\sum_{\lambda\in \HH_{a+1}}x^{\ell(\lambda)}q^{|\lambda|}
&=
\sum_{s,t,u,v,w\geq 0}\frac{q^{Q_1(s,t,u,v,w)+2as+(6a-4)t+(6a-2)u+(2a-2)v+(4a-4)w}}{(q^2;q^2)_s(q^6;q^6)_t(q^6;q^6)_u(q^2;q^2)_v(q^2;q^2)_w}x^{s+3t+3u+v+2w},\\
\sum_{\lambda\in \NAN_{2a-1}}x^{\ell(\lambda)}q^{|\lambda|}
&=
\sum_{s,t,u,v\geq 0}\frac{q^{Q^{(1)}_2(s,t,u,v)+(2a-1)s+(2a-2)u+(4a-4)v}}{(q^2;q^2)_s(q^2;q^2)_t(q^2;q^2)_u(q^2;q^2)_v}x^{s+t+2u+2v}.
\end{align*}
\end{Prop}

\begin{proof}
As usual, one can complete the proof
by comparing $q$-difference equations on both sides.
On the left hand side, one can automatically calculate them
by ~\cite{An0,TT} (See also ~\cite[\S4.3]{BJM} and
~\cite[Proposition 4.2]{TT}).
On the right hand side,
the algorithm reviewed in ~\cite[\S7.2 (Step 1),(Step 2),(Step 3)]{Tsu3} gives a nontrivial $q$-difference
equation by putting $(S,D)=(S_{2,2},D_{1}),(S_{4,2,2},D_{3,3}),(S_{6,3,3,2,2},D_{3,3,1,1}),(S_{6,2,3,2},D_{3,1,3})$, respectively
(for $H_{a+1}$, we apply the algorithm to the sum $\sum_{v,s,w,t,u\geq 0}...$ instead of $\sum_{s,t,u,v,w\geq 0}...$), where
$Y_{a_1,...,a_b}=\{(c_1,\dots,c_b)\in\mathbb{Z}^b\mid 0\leq c_j< a_j\}$ for $Y\in\{S,D\}$.
\end{proof}

\begin{Rem}\label{recov}
The right hand side in Proposition
\ref{computerproof} for $x=1$ is nothing but the left hand side
of
Theorem \ref{MAINA2oddN3} for $n=1$, $\varepsilon=1$, $b=2(a-1)$,
Theorem \ref{MAINA2oddN3} for $n=1$, $\varepsilon=0$, $b=2(a-1)$,
Conjecture \ref{MAINA2oddConj} for $n=1$, $b=2(a-1)$,
Theorem \ref{MAINA2oddN3} for $n=2$, $\varepsilon=1$, $b=2a-1$, respectively. We display each quadratic term below.
\begin{align*}
  Q^{(1)}_1(i,j_1)
  &=
  i^2+2j_1^2+2ij_1,\\
  Q^{(0)}_{1}(i,j_1,j_2)
  &=
  2i^2+2j_1^2+4j_2^2+2ij_1+4ij_2+4j_1j_2,\\
  Q_1(i,j,k,\ell_1,\ell_2)
  &=
  i^2+9j^2+9k^2+2\ell_1^2+4\ell_2^2\\
  &\quad
  +6ij+6ik+2i\ell_1+6i\ell_2+18jk+6j\ell_1+12j\ell_2+6k\ell_1+12k\ell_2+4\ell_1\ell_2,\\
Q^{(1)}_{2}(i,j_1,j_2,j_3)
&=
2i^2+2j_1^2+4j_2^2+6j_3^2+2ij_1+4ij_2+6ij_3+4j_1j_2+4j_1j_3+8j_2j_3.
\end{align*}
\end{Rem}

\begin{Rem}\label{KRanotherproof}
  When $n=1$ and $b=0$ (resp. $b=2$), Conjecture \ref{MAINA2oddConj} is equivalent to
  the partition theorem $\sum_{\lambda\in\HH_2}q^{|\lambda|}=\frac{(q^6;q^{12})_{\infty}}{(q^2,q^3,q^4,q^8,q^9,q^{10};q^{12})_{\infty}}$
  (resp. $\HH_3\PT T^{(12)}_{4,5,6,7,8}$) in \eqref{KR12} by the third identity in Proposition \ref{computerproof}
  for $a=1$ (resp. $a=2$). 
  As noted in ~\cite[Theorem 2, Theorem 3]{KR2},
  the partition theorem $\HH_1\PT T^{(12)}_{1,4,6,8,11}$ in \eqref{KR12} follows from these.
\end{Rem}

\hspace{0mm}

\noindent{\bf Acknowledgments.}
We are grateful to Ole Warnaar for helpful discussions, especially for bringing the papers ~\cite{War,War2} to our attention.

\end{document}